\documentclass[letterpaper]{article}

\usepackage[letterpaper,left=1.5in,right=1.5in,top=1.5in,bottom=1.5in]{geometry}

\newcommand{\myauthor}{Benjamin Antieau}
\newcommand{\mytitle}{A reconstruction theorem for abelian categories of twisted sheaves}
\newcommand{\pdftitle}{\mytitle}
\author{\myauthor}

\title{\mytitle}

\usepackage[pdfstartview=FitH,
            pdfauthor={\myauthor},
            pdftitle={\pdftitle},
            colorlinks,
            linkcolor=reference,
            citecolor=citation,
            urlcolor=e-mail,
            ]{hyperref}

\usepackage{mathrsfs}
\usepackage[mathscr]{euscript}
\usepackage{amsmath}
\usepackage{amscd}
\usepackage{amsbsy}
\usepackage{amssymb}
\usepackage{microtype}
\usepackage{enumerate}
\usepackage[all,cmtip]{xy}
\usepackage{tikz}
\usetikzlibrary{matrix,arrows}
\usepackage{dsfont}
\usepackage[abbrev,lite]{amsrefs}
\usepackage{amsthm}
\usepackage{mathpazo}

\usepackage{color}
\definecolor{todo}{rgb}{1,0,0}
\definecolor{conditional}{rgb}{0,1,0}
\definecolor{e-mail}{rgb}{0,.40,.80}
\definecolor{reference}{rgb}{.20,.60,.22}
\definecolor{mrnumber}{rgb}{.80,.40,0}
\definecolor{citation}{rgb}{0,.40,.80}

\DeclareMathOperator{\Ho}{Ho}

\newcommand{\sperf}{\mathrm{sperf}}

\DeclareMathOperator{\op}{op}

\DeclareMathOperator{\im}{im}
\DeclareMathOperator{\St}{St}

\DeclareMathOperator*{\colim}{colim}

\newcommand{\rwe}{\tilde{\rightarrow}}
\newcommand{\riso}{\overset\simeq\rightarrow}
\newcommand{\we}{\simeq}
\newcommand{\tw}{\mathrm{tw}}
\newcommand{\iso}{\cong}

\newcommand{\Gm}{\mathds{G}_{m}}

\DeclareMathOperator{\add}{add}
\newcommand{\supp}{\mathrm{supp}}

\newcommand{\caldararu}{C\u{a}ld\u{a}raru}

\newcommand{\ex}{\mathrm{ex}}
\newcommand{\tors}{\mathrm{tors}}

\newcommand{\qc}{\mathrm{qc}}

\newcommand{\Mod}{\mathrm{Mod}}
\newcommand{\Ch}{\mathrm{Ch}}

\newcommand{\QC}{\mathrm{QC}}
\newcommand{\Coh}{\mathrm{Coh}}
\newcommand{\QCoh}{\mathrm{QCoh}}

\newcommand{\Aff}{\mathrm{Aff}}

\newcommand{\Cat}{\mathrm{Cat}}

\DeclareMathOperator{\Spec}{Spec}
\DeclareMathOperator{\Sp}{Sp}

\DeclareMathOperator{\ann}{ann}

\DeclareMathOperator{\Hoh}{H}

\newcommand{\Hom}{\mathrm{Hom}}
\newcommand{\Fun}{\mathrm{Fun}}

\newcommand{\StQCoh}{\mathscr{QC}\mathrm{oh}}
\newcommand{\StCoh}{\mathscr{C}{\mathrm{oh}}}
\newcommand{\StCh}{\mathscr{C}{\mathrm{h}}}

\newcommand{\et}{\mathrm{\acute{e}t}}

\newcommand{\Zar}{\mathrm{Zar}}

\DeclareMathOperator{\Br}{Br}
\DeclareMathOperator{\dBr}{dBr}

\newcommand{\Nrm}{\mathrm{N}}

\newcommand{\Drm}{\mathrm{D}}

\newcommand{\Crm}{\mathrm{C}}

\newcommand{\Oscr}{\mathscr{O}}

\newcommand{\Ascr}{\mathscr{A}}

\newcommand{\Dscr}{\mathscr{D}}

\newcommand{\Xscr}{\mathscr{X}}

\newcommand{\Yscr}{\mathscr{Y}}

\newcommand{\ZZ}{\mathds{Z}}

\theoremstyle{plain}
\newtheorem{theorem}{Theorem}[section]
\newtheorem{lemma}[theorem]{Lemma}
\newtheorem{proposition}[theorem]{Proposition}

\newtheorem{corollary}[theorem]{Corollary}

\theoremstyle{definition}
\newtheorem{definition}[theorem]{Definition}

\newtheorem{remark}[theorem]{Remark}

\let\oldmarginpar\marginpar
\renewcommand\marginpar[1]{\-\oldmarginpar[\raggedleft\footnotesize #1]%
{\raggedright\footnotesize #1}}

\begin{document}
\maketitle

\begin{abstract}

    \noindent
    We use an idea of Rosenberg to prove a reconstruction theorem for abelian categories of
    $\alpha$-twisted quasi-coherent sheaves on quasi-compact and quasi-separated schemes $X$ when
    $\alpha\in\Br(X)$. By applying the work of To\"en on derived Azumaya algebras, we give a
    proof of \caldararu's conjecture.


    \paragraph{Mathematics Subject Classification 2010}
    Primary:
    \href{http://www.ams.org/mathscinet/msc/msc2010.html?t=14Fxx&btn=Current}{14F22},
    \href{http://www.ams.org/mathscinet/msc/msc2010.html?t=16Hxx&btn=Current}{16H05}.\\
    Secondary:
    \href{http://www.ams.org/mathscinet/msc/msc2010.html?t=18Exx&btn=Current}{18E10},
    \href{http://www.ams.org/mathscinet/msc/msc2010.html?t=14Fxx&btn=Current}{14F05}.

\end{abstract}

\section{Introduction}

Twisted sheaves arise in the study of moduli of sheaves
where the universal sheaf exists in general only as a twisted sheaf.
Given a $\Gm$-gerbe $\Xscr\rightarrow X$, an $\Xscr$-twisted quasi-coherent sheaf
on $X$ is a quasi-coherent sheaf on $\Xscr$ such that the inertial action of $\Gm$ agrees
with the action induced by $\Oscr_\Xscr$ (see~\cite{lieblich-moduli} or Section~\ref{sec:abtwist}).
Recall that $\Gm$-gerbes are classified by the \'etale cohomology group
$\Hoh^2_{\et}(X,\Gm)$, in which there are two special subgroups: the Brauer group $\Br(X)$
and the cohomological Brauer group $\Br'(X)=\Hoh^2_{\et}(X,\Gm)_{\tors}$. The Brauer group
consists of the $\Gm$-gerbes $\Xscr\rightarrow X$ for which there is a non-trivial finite
rank $\Xscr$-twisted vector bundle. On quasi-compact schemes, $\Br(X)\subseteq\Br'(X)$. In
most cases of interest, for instance when $X$ possesses an ample line bundle,
$\Br(X)=\Br'(X)$. In the sequel, we will speak of $\alpha$-twisted quasi-coherent sheaves in
reference to $\Xscr$-twisted quasi-coherent sheaves, where $\alpha$ is the class of $\Xscr$
in $\Hoh^2_{\et}(X,\Gm)$.

We prove the following result, a generalization of \caldararu's
conjecture~\cite{caldararu-elliptic}*{Conjecture 4.1}.

\begin{theorem}[\caldararu's conjecture]
    Let $X$ and $Y$ be quasi-compact and quasi-separated schemes over a commutative ring $R$
    (for instance over $\ZZ$), and fix $\alpha\in\Br(X)$ and $\beta\in\Br(Y)$. Suppose that
    there is an equivalence $\QCoh(X,\alpha)\rwe\QCoh(Y,\beta)$ of $R$-linear abelian
    categories of quasi-coherent twisted sheaves. Then, there exists an isomorphism
    $f:X\rightarrow Y$ of $R$-schemes such that $f^*(\beta)=\alpha$.
\end{theorem}

We also prove a similar result for noetherian schemes $X$ and $Y$ where we only require
$\alpha\in\Br'(X)$ and $\beta\in\Br'(Y)$. Our argument relies on a reconstruction
theorem of Perego~\cite{perego}.

\caldararu's conjecture has been established previously by Canonaco and
Stellari~\cite{canonaco-stellari} when $X$ and $Y$ are smooth projective varieties over an
algebraically closed field.
The methods of Canonaco and Stellari use Fourier-Mukai functors and pass through derived
categories. Our methods also use derived categories, or rather their dg categorical
enhancements.

The proof has three main steps. First, we show that $X$ as in the theorem can be
reconstructed from the abelian category $\QCoh(X,\alpha)$. This is an extension of a theorem
of Rosenberg~\cite{rosenberg} and has its roots in Gabriel's thesis~\cite{gabriel}, where
the statement for abelian categories of untwisted quasi-coherent sheaves was proved for
noetherian schemes. Second, we show that the reconstruction theorem results in an
isomorphism $f:X\rightarrow Y$ and an equivalence $\StQCoh(\alpha)\we\StQCoh(f^*(\beta))$ of
Zariski stacks of abelian categories on $X$. It remains to prove that the existence of this
equivalence of stacks implies that
$f^*(\beta)=\alpha$. We do this by inducing an equivalence
$\Dscr^{\et}_{dg}(\alpha)\we\Dscr^{\et}_{dg}(f^*(\beta))$ of \'etale stacks of dg categories. But, To\"en showed that the
derived Brauer group $\dBr(X)=\Hoh^1_{\et}(X,\ZZ)\times\Hoh^2_{\et}(X,\Gm)$ classifies
stacks of dg categories that are \'etale locally derived Morita equivalent to the base.
Since $\Br(X)\subseteq\dBr(X)$, the theorem follows. Instead of using the results of To\"en, we could also
use stable $\infty$-categories and the work of~\cite{ag}.

It is worth pointing out another recent reconstruction theorem, proven by
Lurie~\cite{lurie-tannaka}. Lurie shows that $\otimes$-functors from $\QCoh(Y)$ to $\QCoh(X)$
correspond to maps $X\rightarrow Y$ when $X$ is a scheme and $Y$ is a geometric (Artin) stack. One might attempt to apply this
in our situation, but $\QCoh(X,\alpha)$ for instance is not a $\otimes$-category. It is
natural then to attempt to replace $\QCoh(X,\alpha)$ by $\QCoh(\Xscr)$, where
$\Xscr\rightarrow X$ is the $\Gm$-gerbe corresponding to $\alpha$. This is a
$\otimes$-category, and it decomposes as
\begin{equation*}
    \QCoh(\Xscr)\we\prod_{k\in\ZZ}\QCoh(X,\alpha^k).
\end{equation*}
However, in order to prove our theorem using a technique like this, we would need to pass from
an equivalence $\QCoh(X,\alpha)\we\QCoh(Y,\beta)$ to an equivalence
$\QCoh(X,\alpha^k)\we\QCoh(Y,\beta^k)$ for all integers $k$. We do not see how to
make this jump, at least without using \caldararu's conjecture. But, Lurie's result does say
that if $\QCoh(\Xscr)\we\QCoh(\Yscr)$ is a $\otimes$-equivalence respecting the character
decompositions on both sides, then $\Xscr\iso\Yscr$, where $\Yscr\rightarrow Y$ is
the $\Gm$-gerbe on $Y$ associated to $\beta$. Indeed, the natural map
$\QCoh(X)\rightarrow\QCoh(\Xscr)$ is a $\otimes$-functor; composition then gives a
$\otimes$-functor $\QCoh(X)\rightarrow\QCoh(\Yscr)$, which induces a map
$f:X\rightarrow\Yscr\rightarrow Y$. By reversing the argument, we see that $f$ is an
isomorphism. It follows that
$f^*\Yscr\we\Xscr$ as $\Gm$-gerbes, so that $\alpha=f^*\beta$. Thus, with the stronger
hypothesis that there is a $\otimes$-equivalence $\QC(\Xscr)\we\QC(\Yscr)$ respecting
character decompositions, our theorem goes
through, using~\cite{lurie-tannaka}.

\paragraph{Acknowledgments.} We thank the referee for their comments, which resulted in
a more precise exposition.

\section{Background on abelian categories}

In this section, we review some basic concepts from the theory of abelian categories. These
ideas are due to Gabriel, Grothendieck, and Serre. For details and proofs,
see~\cite{gabriel} or~\cite{grothendieck-tohoku}.

We consider two types of abelian categories, those that behave like the abelian category of
quasi-coherent sheaves $\QCoh(X)$ when $X$ is
quasi-compact and quasi-separated, and those that behave like the abelian category of
coherent sheaves $\Coh(X)$ when $X$ is noetherian. These are ``big''
and ``small'' abelian categories, respectively.

\begin{definition}
    Let $A$ be an abelian category. A non-empty full subcategory $B\subseteq A$ is thick (or
    \'epaisse) if it is closed under taking subobjects, quotients, and extensions; in other
    words, if for every exact sequence
    \begin{equation*}
        0\rightarrow M'\rightarrow M\rightarrow M''\rightarrow 0
    \end{equation*}
    in $A$, we have $M\in B$ if and only if $M'$ and $M''$ are in $B$.
\end{definition}

If $B\subseteq A$ is thick, then following Serre we can often define a quotient abelian category
$A/B$, which has as objects the same objects as $A$ and where
\begin{equation}\label{eq:serre}
    \Hom_{A/B}(M,N)=\colim\Hom_A(M',N/N')
\end{equation}
where the colimit is over all $M'\subseteq M$ and $N'\subseteq N$ where $M/M'\in B$ and
$N'\in B$ and is taken in the abelian category of abelian groups. In order for the colimit
to exist in $\Mod_\ZZ$ (in our fixed universe), we need to know that the colimit diagram is
essentially small. This is guaranteed if either $A$ is essentially small (as is the case for
$\Coh(X)$ when $X$ is abelian) or is a
Grothendieck abelian category (such as $\QCoh(X)$ for $X$ quasi-compact and quasi-separated), a concept we now define.

\begin{definition}
    An abelian category $A$ is said to satisfy AB3 if it has all small coproducts (this
    implies that $A$ has all colimits). An
    abelian category $A$ satisfies AB5 if it satisfies AB3 and if the following
    condition holds: whenever
    $M=\cup_{i=1}^\infty M_i$ where $M_1\subseteq M_2\subseteq\cdots$, the natural map
    $\Hom(M,N)\rightarrow\lim\Hom(M_i,N)$ is an isomorphism for all $N$. A generator of an abelian
    category is an object $U$ such that if $N\subseteq M$ is a proper subobject, there is a
    morphism $U\rightarrow M$ that does not factor through $N$. An abelian category is
    called a Grothendieck abelian category if it satisfies AB5 and has a generator.
\end{definition}

If $A$ possesses a generator, then every object of $A$ has a set of subobjects and quotient
objects. It follows that for any thick subcategory $B\subseteq A$, the colimit appearing
in~\eqref{eq:serre} is over a small diagram, and hence exists in $\Mod_\ZZ$. This will
guarantee that all quotients taken in this paper are well-defined without expanding the
universe.

\begin{proposition}
    Let $B\subseteq A$ be a thick subcategory, and assume that $B/A$ exists. Then,
    \begin{enumerate}
        \item   $A/B$ is abelian;
        \item   the natural map $A\rightarrow A/B$ is exact.
    \end{enumerate}
\end{proposition}

\begin{definition}
    A thick subcategory $B\subseteq A$ is localizing if the functor $j:A\rightarrow A/B$
    admits a right adjoint $j_*$.
\end{definition}

\begin{lemma}
    If $B\subseteq A$ is localizing, then the right adjoint $j_*$ is fully faithful and
    left exact.
    \begin{proof}
        The fully faithfulness of $j_*$ can be seen by applying~\eqref{eq:serre} to two
        objects in the image of $j_*$: they have no non-zero subquotients contained in
        $B$. Since $j_*$ is a right adjoint, it preserves limits that exist. Since $A$
        and $A/B$ have finite limits, it preserves finite limits. But, kernels are finite limits,
        hence $j_*$ is left exact.
    \end{proof}
\end{lemma}

The following proposition is very useful for checking that a subcategory of a Grothendieck
abelian category is localizing.

\begin{proposition}\label{prop:localizing}
    Let $A$ be a Grothendieck abelian category, and let $B\subseteq A$ be a thick
    subcategory. Then, the following are equivalent:
    \begin{enumerate}
        \item   $B\subseteq A$ is localizing;
        \item   the inclusion $B\rightarrow A$ admits a right adjoint;
        \item   the inclusion $B\rightarrow A$ preserves colimits;
        \item   every object $M$ of $A$ contains a maximal subobject contained in $B$.
    \end{enumerate}
    \begin{proof}
        This is left to the reader. We remind them of the adjoint functor theorem and the
        fact that left adjoints preserve colimits. For details, see
        Gabriel~\cite{gabriel}*{Section III}.
    \end{proof}
\end{proposition}

\section{Abelian categories of twisted sheaves}\label{sec:abtwist}

Let $X$ be a scheme and let $\alpha\in\Hoh^2(X,\Gm)$ be represented by a $\Gm$-gerbe
$\Xscr\rightarrow X$. Then, we write $\QCoh(X,\alpha)$ for the abelian category of
$\Xscr$-twisted sheaves $\QCoh^{\tw}(\Xscr)$, defined for instance in
Lieblich~\cite{lieblich-moduli}. An $\Xscr$-twisted sheaf is a quasi-coherent sheaf on
$\Xscr$ such that the inertial action of $\Gm$ on the left agrees with the action
through $\Oscr_{\Xscr}$. If $\alpha$
is the Brauer class of an Azumaya algebra $\Ascr$ on $X$, then
$\QCoh(X,\alpha)\we\QCoh(X,\Ascr)$, where $\QCoh(X,\Ascr)$ denotes the abelian category of quasi-coherent left $\Ascr$-modules on
$X$. If $f:Y\rightarrow X$ is an $X$-scheme, write $\QCoh(Y,\alpha)$ for
$\QCoh(Y,f^*(\alpha))$. This defines a prestack $\StQCoh(\alpha)$ of abelian categories over $X$.

\begin{proposition}
    The prestack $\StQCoh(\alpha)$ of $\Xscr$-twisted sheaves forms a stack of abelian
    categories on the big \'etale site over $X$.
    \begin{proof}
        See~\cite{lieblich-moduli}*{Proposition 2.1.2.3}.
    \end{proof}
\end{proposition}

\begin{proposition}
    The abelian category $\QCoh(X,\alpha)$ is a Grothendieck abelian category when $X$ is
    quasi-compact and quasi-separated.
    \begin{proof}
        This follows because the abelian category of all quasi-coherent modules on $\Xscr$
        is a Grothendieck abelian category. The main task is to produce a generator, but
        this can be done by \'etale descent. Or, when $\alpha\in\Br(X)$, using the equivalence
        $\QCoh(X,\alpha)\we\QCoh(X,\Ascr)$, we can produce a generating set by taking a
        representative set of all $\Ascr$-modules of finite type. The direct sum of the
        elements of this set will be a generator.
    \end{proof}
\end{proposition}

\begin{remark}
    The stack $\StQCoh(\alpha)$ is a stack of $\Oscr_X$-linear abelian categories, in the
    sense that for every $U\subseteq X$ there is the structure of $\Gamma(U,\Oscr_X)$-linear
    abelian category on $\StQCoh(U,\alpha)$, and these are compatible with restriction.
\end{remark}

\section{Quasi-coherent reconstruction}

In this section, we prove the first part of the reconstruction theorem.

\begin{theorem}\label{thm:main}
    Suppose that $X$ and $Y$ are quasi-compact and quasi-separated schemes over a
    commutative ring $R$ with Brauer
    classes $\alpha\in\Br(X)$ and $\beta\in\Br(Y)$. If there is an equivalence of $R$-linear abelian
    categories $F:\QCoh(X,\alpha)\rwe\QCoh(Y,\beta)$, then
    there exists a unique isomorphism $f:X\rightarrow Y$ of $R$-schemes compatible with
    $F$ via supports in the sense that if $M$ is an object of $\QCoh(X,\alpha)$, then
    $f(\supp_X(M))=\supp_Y(F(M))$. Moreover, $f$ induces an
    equivalence of stacks of
    $\Oscr_X$-linear abelian categories $\StQCoh(\alpha)\we\StQCoh(f^*(\beta))$ on $X$.
\end{theorem}

\begin{corollary}
    Suppose that $X$ and $Y$ are quasi-compact and quasi-separated schemes such that
    $\QCoh(X)\we\QCoh(Y)$ as abelian categories. Then, $X\iso Y$.
\end{corollary}

\begin{remark}
    The corollary is originally due to Rosenberg~\cite{rosenberg}, although there are some problems
    with the published exposition, as it claims to give the result for arbitrary schemes.
\end{remark}

To prove the theorem we introduce an intermediary construction, the spectrum of
$\QCoh(X,\alpha)$, following Rosenberg~\cite{rosenberg}. While our definition differs from
Rosenberg's, the main ideas of the proof of the critical Proposition~\ref{prop:bijection}
below are due to Rosenberg, with some important adjustments to take into account the
twisting.

\begin{definition}
    If $A$ is an abelian category and $M\in A$, let $\add M$ denote the full subcategory of
    $A$ consisting of all subquotients of finite direct sums of the object $M$. Call an
    abelian category $A$ quasi-local if it has a quasi-final object, which is an object
    contained in $\add M$ for all non-zero objects $M$ of $A$.
\end{definition}

\begin{definition}
    If $A$ is an abelian category, let $\Sp A$ denote the class of localizing subcategories
    $B\subseteq A$ such that $A/B$ is quasi-local. Denote by $S_B$ a quasi-final object of
    $A/B$, which we often view as an object of $A$ via the right adjoint to $A\rightarrow
    A/B$.
\end{definition}

\begin{lemma}\label{lem:commuting}
    Suppose that $j:A\rightarrow A/C$ is an exact localization and let $B$ be a localizing
    subcategory in $\Sp A$. Then, either $j(S_B)=0$ or $B/B\cap C\in\Sp A/C$.
    \begin{proof}
        We assume that $j(S_B)\neq 0$.
        Note that the localization $A\rightarrow (A/C)/(B/B\cap C)$ kills every object of
        $B$. Therefore, it factors through $A\rightarrow A/B$. Thus, the image of $S_B$ in
        $(A/C)/(B/B\cap C)$ is quasi-final.
    \end{proof}
\end{lemma}

The first result is that there is a bijection between $\Sp\QCoh(X,\alpha)$ and $X$. In particular, it is
a set rather than a proper class. We will use implicitly throughout the equivalence of
$\QCoh(X,\alpha)$ and $\QCoh(X,\Ascr)$ when $\Ascr$ is an Azumaya algebra with class
$\alpha$.

\begin{proposition}\label{prop:bijection}
    If $X$ is quasi-compact and quasi-separated, $\alpha\in\Br(X)$, and $\Ascr$ is any
    Azumaya algebra representing $\alpha$, there are inverse
    bijections
    \begin{equation*}
        \xymatrix{
            \Sp\QCoh(X,\alpha)\ar@<.5ex>[r]^-\psi    &   X\ar@<.5ex>[l]^-\phi
        }
    \end{equation*}
    defined by
    \begin{gather*}
        \psi(B)=\text{the generic point of $\supp(S_B)$},\\
        \phi(x)=\{M\in\QCoh(X,\alpha)|\Ascr(x)\notin\add M\},
    \end{gather*}
    where $\Ascr(x)=\Ascr/p(x)$ and $p(x)$ is the kernel of
    $\Ascr\rightarrow\Ascr\otimes_{\Oscr_X}k(x)$.
    Moreover, $\phi(x)$ does not depend on the Azumaya algebra $\Ascr$.
    \begin{proof}
        Note that we compute $\supp(S_B)\subseteq X$ by viewing $S_B$ as an object of
        $\QCoh(X,\alpha)$ via the fully faithful right adjoint
        $\QCoh(X,\alpha)/B\rightarrow\QCoh(X,\alpha)$. The main work of the proof is to show
        that both $\psi$ and $\phi$ are in fact well-defined.

        To see that $\psi$ is well-defined, let $B$ be a localizing subcategory of
        $\QCoh(X,\alpha)$. To check that $\supp(S_B)$ is irreducible, we can assume that
        $X=\Spec R$ by Lemma~\ref{lem:commuting}.
        Then, it suffices to check that $\ann_R(S_B)$ is prime. Suppose that
        $a,b\in R$ are elements such that $ab\in\ann_R(S_B)$. Consider the submodules $aS_B$
        and $bS_B$ of $S_B$. Assume that $aS_B$ is non-zero. Then, $b\in\ann_R(aS_B)$. Since $S_B$
        is in $\add(aS_B)$ as $S_B$ is quasi-final, it follows that $b\in\ann_R(S_B)$. So,
        $\ann_R(S_B)$ is prime. Finally, it is clear than any two quasi-final objects in
        $\QCoh(X,\alpha)/B$ have the same support. Thus, $\psi$ is well-defined.

        The map $\phi(x)$ defines a full subcategory of $A$, but it is not
        immediately clear that it is thick much less localizing and contained in $\Sp\QCoh(X,\alpha)$.
        Let $x\in X$, and take $\Ascr(x)$ as above. Then, $\Ascr(x)$ is the quotient of $\Ascr$
        by a sheaf of two-sided prime ideals. Let $N\subseteq\Ascr(x)$ be a non-zero left
        submodule. We want to show that $\Ascr(x)\in\add N$.

        To do this, we can assume that
        $N$ is of finite type. Let $j:U\subseteq X$ be an affine open with $j$ quasi-compact and such
        that $x\in U$ and $N|_U$ is non-zero. Then, $N|_U\subseteq\Ascr(x)|_U$. The ideal
        \begin{equation*}
            \mathrm{ann}_{\Ascr(x)(U)}(N(U))=\{x\in\Ascr(x)(U)\,|\,xN(U)=0\}
        \end{equation*}
        is a $2$-sided ideal. Since $\Ascr(x)\otimes_{\Oscr_X}k(y)$ is a central simple
        algebra for $y\in\overline{\{x\}}$, it follows that
        $\mathrm{ann}_{\Ascr(x)(U)}(N(U))\otimes_{\Oscr_X}k(y)=0$ for all such $y$. Since $N$ is
        of finite type, so is $\mathrm{ann}_{\Ascr(x)(U)}(N(U))$. Nakayama's lemma now implies
        that the annihilator $\mathrm{ann}_{\Ascr(x)(U)}(N(U))$ vanishes. Pick generators
        $a_1,\ldots,a_n$ of $N(U)$. Then, the natural map
        \begin{equation*}
            \Ascr(x)(U)\rightarrow\bigoplus_{i=1}^n \Ascr(x)(U)/\mathrm{ann}_{\Ascr(x)(U)}(a_i)
        \end{equation*}
        is injective. But, each quotient $\Ascr(x)(U)/\mathrm{ann}_{\Ascr(x)(U)}(a)$ is contained in
        $\add(N(U))$ (viewed as a subcategory of the abelian category $\Mod_{\Ascr(x)(U)}$). It follows that $\Ascr(x)(U)$ is contained in $\add(N)(U)$.
        This means that $\Ascr(x)|_U\in\add N|_U$ in $\QCoh(U,\alpha)$. By adjunction,
        $j_*\Ascr(x)|_U\in\add N$.

        Now, we claim that $\Ascr(x)$ is a subsheaf of a direct sum
        \begin{equation*}
            j_{i,*}\Ascr(x)|_{U_i}
        \end{equation*}
        for $j_i:U_i\rightarrow X$ finitely many quasi-compact open immersions. This will
        show that $\Ascr(x)$ is contained in $\add N$. Write $V$ for the support of
        $\Ascr(x)$. That is $V=\overline{\{x\}}$. Then, the inclusion $i:V\rightarrow X$ is
        quasi-compact. It follows that there are finitely many open affines $U_i$ of $X$ each
        intersecting $V$ that cover $V$ and such that $j_{i}:U_i\rightarrow X$ is
        quasi-compact. The claim follows, and we have proven that $\Ascr(x)\in\add N$, as
        desired.

        Now, it follows that $\phi(x)$ is thick. Indeed, let
        \begin{equation*}
            0\rightarrow M'\rightarrow M\rightarrow M''\rightarrow 0
        \end{equation*}
        be an exact sequence. If $\Ascr(x)\notin\add M$, then it is clearly not in $\add M'$
        or in $\add M''$. Conversely, if it is not in $\add M'$ nor in $\add M''$, then it
        cannot be in $\add M$. Otherwise, $\Ascr(x)$ would be a subquotient of $M^{\oplus
        n}$ for some $n$. Either an object of $\add M'$ has a non-zero map to $\Ascr(x)$ in
        which case $\Ascr(x)\in\add M'$ by the previous paragraph, or there is no such map, in which case
        $\Ascr(x)\in\add M''$. Therefore, $\phi(x)$ is thick. If $M=\colim_i M_i$ for some
        objects $M_i\in\phi(x)$ and if $\Ascr(x)\in\add M$, then some object of $\add M_i$
        must map to $\Ascr(x)$ with non-zero image by axiom AB5. Hence, $\phi(x)$ is a
        localizing subcategory of $\QCoh(X,\alpha)$ by Proposition~\ref{prop:localizing}. To
        see that the quotient is quasi-local, note that the image of $\Ascr(x)$ in the
        quotient $\QCoh(X,\alpha)/\phi(x)$ is quasi-final by the definition of $\phi(x)$.

        Now that we have seen that the maps are well-defined, we show that they are mutual
        inverses. It is clear that $\psi(\phi(x))=x$. So, fix $B\in\Sp\QCoh(X,\alpha)$.
        Then, $\phi(\psi(B))$ consists of $M\in\QCoh(X,\alpha)$ such that
        $\Ascr(x)\notin\add M$ where $x$ is the generic point of the support of $S_B$.
        Clearly, $B\subseteq\phi(\psi(B))$. Fix a quasi-final object $S'$ of
        $\QCoh(X,\alpha)/\phi(\psi(B))$ (viewed as usual as an object of $\QCoh(X,\alpha))$.
        Let $M\in\phi(\psi(B))$ be an object not contained in $B$. Then, by definition, 
        $S'\notin\add M$, and moreover, $S_B\in\add M$. Now, by construction of $\phi$ and
        $\psi$, the image of $S_B$ in $\QCoh(X,\alpha)/\phi(\psi(x))$ is non-zero. By
        quasi-finality of $S'$, we have that $S'\in\add S_B$. But, this means that $S'\in\add
        M$, in contradiction to our choice of $M$. Therefore, $B=\phi(\psi(B))$.

        Finally, that $\phi(x)$ does not depend on the choice of Azumaya algebra
        $\Ascr$ with Brauer class $\alpha$ follows from the fact that $\phi$ is the inverse
        to $\psi$, and hence is unique.
    \end{proof}
\end{proposition}

\begin{remark}
    Let $X$ be an arbitrary quasi-compact and quasi-separated scheme and
    $\alpha\in\Br'(X)=\Hoh^2_{\et}(X,\Gm)_{\tors}$ a cohomological Brauer class. One might
    ask whether the theorem extends to this case. At the moment, we are not certain,
    although we can
    say the following. Note that $\psi_X$ is well-defined regardless.
    Let $U\subseteq X$ be an affine open subset. Then, by a theorem of
    Gabber~\cite{gabber}*{Chapter II, Theorem 1}, the restriction of $\alpha$ to $U$ is represented by an Azumaya
    algebra. Looking at the commutative diagrams
    \begin{equation*}
        \xymatrix{
            \Sp\QCoh(U,\alpha)\ar[r]^-{\psi_U}\ar[d]     &   U\ar[d]\\
            \Sp\QCoh(X,\alpha)\ar[r]^-{\psi_X}           &   X
        }
    \end{equation*}
    for all such $U$,
    we see that $\psi_X$ is surjective. The difficulty is in constructing the
    inverse to $\psi_X$, where we needed, or at least used, an Azumaya algebra.
\end{remark}

The proposition says that the set of points of $X$ can be recovered from $\QCoh(X,\alpha)$.
We go even farther, showing that $X$ can be recovered as a ringed space from
$\QCoh(X,\alpha)$. To do this, we first introduce a topology on $\Sp\QCoh(X,\alpha)$.

\begin{definition}
    If $A$ is a Grothendieck abelian category, and if $M$ is an object of $A$, define
    $\supp_{\Sp}(M)\subseteq\Sp A$ to be the set of all $B\in\Sp A$ such that the image of
    $M$ in $A/B$ is non-zero. Give $\Sp A$ the topology generated by the basis of closed sets
    \begin{equation*}
        \{\supp_{\Sp}(M)|\text{$M\in A$ is finitely presented}\}.
    \end{equation*}
    Recall that $M$ is finitely presented if the functor $\Hom_A(M,-)$ commutes with
    arbitrary coproducts. Because finite direct sums of finitely presented objects are
    finitely presented, the set is closed under finite unions.
\end{definition}

\begin{proposition}
    The maps
    \begin{equation*}
        \xymatrix{
            \Sp\QCoh(X,\alpha)\ar@<.5ex>[r]^-\psi    &   X\ar@<.5ex>[l]^-\phi
        }
    \end{equation*}
    are inverse homeomorphisms.
    \begin{proof}
        It suffices to show that $\psi$ induces a bijection between elements of a
        basis for the topologies of $\Sp\QCoh(X,\alpha)$ and $X$. Let $V\subseteq X$ be a
        closed subset defined by a sheaf of ideals $I_V\subseteq\Oscr_X$ of finite type.
        Then, $\Oscr_V$ is a finitely presented $\Oscr_X$-module. Write $\Ascr_V$ for
        $\Ascr\otimes_{\Oscr_X}\Oscr_V$. It is a left $\Ascr$-module of finite presentation
        in $\QCoh(X,\alpha)$. It follows that
        $\supp_{\Sp}(\Oscr_V)$ as defined above is a closed subset of $\Sp\QCoh(X,\alpha)$.
        We show that $\psi(\supp_{\Sp}(\Ascr_V))=\supp(\Ascr_V)=V$. Since the closed subsets
        defined by a finitely generated sheaf of ideals $I_V$ of finite type form a basis
        for $X$, the proposition will follow. Now, if $B\in\supp_{\Sp}(\Ascr_V)$, then
        $S_B\in\add\Ascr_V$,
        where $S_B$ is a quasi-final object of the quotient
        $\QCoh(X,\alpha)/B$. It follows that $\supp(S_B)\subseteq\supp(\Ascr_V)=V$. Hence,
        $\psi(\supp_{\Sp}(\Ascr_V))\subseteq V$.
        On the other hand, if $x\in V$, then recall that $\Ascr(x)$ is precisely
        $\Ascr_{\overline{\{x\}}}$, where $\overline{\{x\}}$ is the closure of $\{x\}$.
        Since $\overline{\{x\}}$ is a closed subscheme of $V$, it follows that $\Ascr(x)$ is a quotient
        of $\Ascr_V$. That is, $\Ascr(x)\in\add\Ascr_V$. Therefore, $\phi(x)$ does not
        contain $\Ascr_V$. But, this means that $\phi(x)\in\supp_{\Sp}(\Ascr_V)$. Since
        $\phi$ and $\psi$ are inverse bijections, it follows that
        $\psi(\supp_{\Sp}(\Ascr_V))=V$.
    \end{proof}
\end{proposition}

\begin{definition}
    The canonical prestack on $\Sp A$ is the prestack $\St_A^p$ of abelian categories given by sending an
    open set $U\subseteq\Sp A$ to
    \begin{equation*}
        \St_A^p(U)=A/\cap_{B\in U}B,
    \end{equation*}
    where the intersection $\cap_{B\in U}B$ is localizing as each $B$ is localizing.
    Recall that any abelian category $A$ has a center $C(A)$, which is the commutative ring of
    endomorphisms of the identity functor of $A$.
    By taking the center of these categories, we obtain a presheaf of commutative rings
    $\Oscr_A^p$, which has $\Oscr_A^p(U)=C(\St_A(U))$. Write $\Oscr_A$ for the
    sheafification of $\Oscr_A^p$. Similarly, write $\St_A$ for the stackification of
    $\St_A^p$. Note that $\St_A$ is naturally an $\Oscr_A$-linear stack.
\end{definition}

\begin{proposition}\label{prop:quotients}
    If $U\subseteq\Sp\QCoh(X,\alpha)$ is a quasi-compact immersion, there is a natural equivalence
    \begin{equation*}
        \QCoh(X,\alpha)/\cap_{B\in U}B\we\QCoh(\psi(U),\alpha).
    \end{equation*}
    \begin{proof}
        It suffices to check that $\cap_{B\in U}B$ is equal to $\QCoh_Z(X,\alpha)$ where
        $Z=X-\psi(U)$. Indeed, since $j:\psi(U)\rightarrow X$ is quasi-compact, there is a
        pushforward $j_*:\QCoh(U,\alpha)\rightarrow\QCoh(X,\alpha)$, right adjoint to
        $j^*:\QCoh(X,\alpha)\rightarrow\QCoh(U,\alpha)$. The kernel of $j^*$ is, by definition, the
        class of objects $M$ such that $j^*(M)\iso 0$. But, these are precisely the
        $\alpha$-twisted quasi-coherent sheaves supported on $Z$.

        So, suppose that $M\in\cap_{B\in U}B$. Note from the proof of
        Proposition~\ref{prop:bijection} that $\Ascr(x)$ is a quasi-final object of
        $\QCoh(X,\alpha)/\phi(x)$. In particular, $\Ascr(x)\notin\add M$ for any
        $x\in\psi(U)$. This implies that $M$ must be supported on $Z$. On the other hand,
        clearly if $M$ is supported on $Z$, then $\Ascr(x)\notin\add M$ for $x\in\psi(U)$.
        Hence, $M\in\cap_{B\in U}B$, as desired.
    \end{proof}
\end{proposition}

\begin{corollary}
    The map $\phi:X\rightarrow\Sp\QCoh(X,\alpha)$ induces an equivalence
    \begin{equation*}
        \StQCoh(\alpha)\we\phi^*\St_{\QCoh(X,\alpha)}
    \end{equation*}
    of Zariski stacks of $\Oscr_X$-linear abelian categories.
    \begin{proof}
        This follows from the proposition as the open subschemes with
        quasi-compact inclusion morphism
        $U\subseteq X$ form a basis for the topology on $X$, as $X$ is quasi-compact and
        quasi-separated.
    \end{proof}
\end{corollary}

Given this equivalence of stacks and the construction of the sheaf of rings on the spectrum
of an abelian category, we obtain the following corollary, which says that $X$ can be
reconstructed from $\QCoh(X,\alpha)$.

\begin{corollary}
    The maps
    \begin{equation*}
        \xymatrix{
            \Sp\QCoh(X,\alpha)\ar@<.5ex>[r]^-\psi    &   X\ar@<.5ex>[l]^-\phi
        }
    \end{equation*}
    are inverse isomorphisms of locally ringed spaces over $R$. In particular,
    $\Sp\QCoh(X,\alpha)$ is a scheme.
    \begin{proof}
        This follows from the construction of the sheaf of rings on $\Sp\QCoh(X,\alpha)$ as
        well as the fact that when $X=\Spec S$ is affine the center of $\QCoh(X,\alpha)$ is
        isomorphic to $S$.
    \end{proof}
\end{corollary}


\begin{proof}[Proof of Theorem~\ref{thm:main}]
    Suppose now that $X$ and $Y$ are quasi-compact and quasi-separated schemes with
    $\alpha\in\Br(X)$ and $\beta\in\Br(Y)$, and suppose that there is an equivalence
    $F:\QCoh(X,\alpha)\rwe\QCoh(Y,\beta)$ of $R$-linear abelian categories. Then,
    $F$ induces a unique isomorphism of $R$-schemes
    \begin{equation*}
        \Sp F:\Sp\QCoh(X,\alpha)\riso\Sp\QCoh(Y,\beta)
    \end{equation*}
    with the property that $(\Sp F)(B)$ is the localizing subcategory $F(B)$ of $\QCoh(Y,\beta)$.
    Indeed, $F$ defines a bijection between the sets of localizing subcategories of
    $\QCoh(X,\alpha)$ and $\QCoh(Y,\beta)$, which then respects the topology on their
    spectra, since it respects supports of objects. This proves the uniqueness part of the
    theorem as well. Finally, $\Sp F$ induces
    an equivalence between prestacks of abelian categories, and hence on the structure
    sheaves. In particular, there is an equivalence
    \begin{equation*}
        \St_{\QCoh(X,\alpha)}\we(\Sp F)^*\St_{\QCoh(Y,\beta)}
    \end{equation*}
    of stacks of $\Oscr_{\QCoh(X,\alpha)}$-abelian categories on $\Sp\QCoh(X,\alpha)$.

    Write $\phi_X$ and $\psi_X$ for the maps of Proposition~\ref{prop:bijection} on $X$, and
    similarly for $\phi_Y$ and $\psi_Y$. Proposition~\ref{prop:quotients} implies that there is an
    equivalence
    \begin{equation*}
        \St_{\QCoh(Y,\beta)}\we\psi_Y^*\StQCoh(\beta)
    \end{equation*}
    of stacks of $\Oscr_{\QCoh(Y,\beta)}$-linear abelian categories on $\Sp\QCoh(Y,\beta)$,
    where we view $\StQCoh(\beta)$ as a Zariski stack on $Y$. Similarly, $\phi_X$ induces an
    equivalence
    \begin{equation*}
        \StQCoh(\alpha)\we\phi_X^*\St_{\QCoh(X,\alpha)}
    \end{equation*}
    of stacks of $\Oscr_X$-linear abelian categories on $X$.

    Letting $f=\psi_Y\circ(\Sp F)\circ\phi_X$,
    we obtain an equivalence
    $\StQCoh(\alpha)\we f^*\StQCoh(\beta)$. If $U\subseteq X$ is an open subscheme, then
    $(f^*\StQCoh(\beta))(U)=\StQCoh(f(U),\beta)$. Thus, by definition,
    $f^*\StQCoh(\beta)\we\StQCoh(f^*(\beta))$. The theorem is proved.
\end{proof}

\section{Noetherian reconstruction}

We remark that the following theorem holds, where the assumption that $\alpha\in\Br(X)$ is
dropped. In the statements below, $\Coh(X,\alpha)$ is the abelian category of
$\alpha$-twisted coherent $\Oscr_X$-modules.

\begin{theorem}\label{thm:maincoherent}
    Suppose that $X$ and $Y$ are noetherian schemes over a
    commutative ring $R$ with cohomological Brauer
    classes $\alpha\in\Br'(X)$ and $\beta\in\Br'(Y)$. If there is an equivalence of $R$-linear abelian
    categories $F:\Coh(X,\alpha)\rwe\Coh(Y,\beta)$, then
    there exists a unique isomorphism $f:X\rightarrow Y$ of $R$-schemes compatible with
    $F$ via supports. Moreover, $f$ induces an equivalence of stacks of
    $\Oscr_X$-linear abelian categories $\StCoh(\alpha)\we\StCoh(f^*(\beta))$ on $X$.
\end{theorem}

The theorem follows immediately from the constructions of Perego~\cite{perego}. Perego
reconstructs $X$ from $\Coh(X,\alpha)$ by classifying the Serre subcategories of
$\Coh(X,\alpha)$, topologizing the resulting set using the geometry of localizations in
$\Coh(X,\alpha)$, and using centers of abelian categories to obtain a sheaf of rings on the
result.

As a consequence, we have the following corollary.

\begin{corollary}\label{cor:maincoherent}
    Suppose that $X$ and $Y$ are noetherian schemes over a
    commutative ring $R$ with cohomological Brauer
    classes $\alpha\in\Br'(X)$ and $\beta\in\Br'(Y)$. If there is an equivalence of $R$-linear abelian
    categories $F:\Coh(X,\alpha)\rwe\Coh(Y,\beta)$, then
    there exists a unique isomorphism $f:X\rightarrow Y$ of $R$-schemes compatible with $F$
    via supports. Moreover, $f$ induces an equivalence of stacks of
    $\Oscr_X$-linear abelian categories $\StQCoh(\alpha)\we\StQCoh(f^*(\beta))$ on $X$.
    \begin{proof}
        This follows from the theorem, since the ind-completion of $\Coh(X,\alpha)$ is
        precisely $\QCoh(X,\alpha)$. Recall that the ind-completion of $\Coh(X,\alpha)$
        is the abelian category of exact functors
        $\Fun^{\ex}(\Coh(X,\alpha)^{\op},\Mod_{\ZZ})$. As ind-completion is purely categorical, the
        equivalence at the level of stacks of coherent twisted sheaves induces an equivalence of
        stacks of quasi-coherent twisted sheaves, by stackifying the pointwise
        ind-completion of $\StCoh(X,\alpha)$ and $\StCoh(Y,\beta)$.
    \end{proof}
\end{corollary}

\section{\caldararu's conjecture}

\caldararu's conjecture now follows from the following theorem, which we prove using results
of To\"en~\cite{toen-derived}, and hence using derived algebraic geometry.

\begin{theorem}
    Suppose that $X$ is a quasi-compact and quasi-separated scheme. Suppose that $\alpha$
    and $\beta$ are in $\Hoh^2_{\et}(X,\Gm)$, and suppose that there is an equivalence
    \begin{equation*}
        \StQCoh(\alpha)\we\StQCoh(\beta)
    \end{equation*}
    of Zariski stacks of $\Oscr_X$-linear abelian categories. Then, $\alpha=\beta$.
    \begin{proof}
        The basic idea of the proof is that we use the hypothesized equivalence of Zariski stacks of
        abelian categories to induce an equivalence
        of \'etale stacks of locally presentable dg categories on $X$. The conclusion will then
        follow from~\cite{toen-derived}*{Corollary 3.12}. For background on stacks of dg
        categories, see~\cite{toen-derived}*{Definition 3.6}.

        The stack $\StQCoh(\alpha)$ is a stack on $\Aff_X^{\Zar}$, the category of Zariski
        open immersions $\Spec S\rightarrow X$ of affine schemes. Below, we will use
        implicitly the ideas of Grothendieck~\cite{sga1}*{Expos\'e VI} and the equivalence
        between the $2$-category of fibered
        categories over $\Aff_X^\Zar$ and the $2$-category of pseudo-functors from $\Aff_X^\Zar$ to the $2$-category of categories.
        Hence, we assume that $\StQCoh(\alpha)$ has been straightened, so that the pullback map
        \begin{equation*}
            f^*:\QCoh(\Spec S,\alpha)\rightarrow\QCoh(\Spec T,\alpha)
        \end{equation*}
        is induced by the tensor product $M\mapsto T\otimes_S M$ for any map $\Spec
        T\xrightarrow{f}\Spec S$ of Zariski open immersions. Given a composition $\Spec
        U\xrightarrow{g}\Spec T\xrightarrow{f}\Spec S$, there is a canonical isomorphism
        \begin{equation*}
            \sigma_{g,f}:U\otimes_T(T\otimes_S M)\riso (U\otimes_S M)
        \end{equation*}
        induced by $u\otimes t\otimes m\mapsto ut\otimes m$. The canonical isomorphism
        $\sigma_{g,f}$ defines a natural isomorphism of functors $\sigma_{g,f}:g^*\circ
        f^*\Rightarrow(g\circ f)^*$. There are uniquely determined higher coherences for
        compositions of the pullback maps $f^*$. For instance, if there is a third map $h:\Spec V\rightarrow \Spec U$,
        then the diagram
        \begin{equation*}
            \xymatrix{
                V\otimes_U(U\otimes_T(T\otimes_S M))\ar[r]_{\sigma_{h,g}(T\otimes_S M)}\ar[d]^{V\otimes_U\sigma_{g,f}}    & (V\otimes_T(T\otimes_S
                M))\ar[d]^{\sigma_{h\circ g,f}}\\
                V\otimes_U(U\otimes_S M))\ar[r]_{\sigma_{h,g\circ f}}              &   (V\otimes_SM)
            }
        \end{equation*}
        commutes automatically. To be precise, we see that the pseudo-functor
        $\Aff_X^{\Zar,\op}\rightarrow\Cat$ extends uniquely to a functor from
        $\Nrm(\Aff_X^{\Zar,\op})$ to the nerve of the $(2,1)$-category $\Cat_{(2,1)}$ consisting of
        categories, functors, and natural isomorphisms.

        Using again the equivalence between fibered categories and pseudo-functors, we can
        assume that $\StQCoh(\alpha)\we\StQCoh(\beta)$ is an equivalence of pseudo-functors.
        Hence, it extends uniquely to an equivalence of functors
        $\Nrm(\Aff_X^{\Zar,\op})\rightarrow\Nrm(\Cat_{(2,1)})$. We take this care only to
        ensure that below, when we create prestacks of dg categories, we get strict
        prestacks where the higher coherence data is induced directly from the coherence
        data for $\StQCoh(\alpha)$ and $\StQCoh(\beta)$.

        Let $\Spec S\rightarrow X$ be a Zariski open immersion, and
        let $\Ch_{dg}^{\sperf}(\QCoh(\Spec S,\alpha))$ denote the dg category of bounded chain
        complexes of finitely presented projective objects of the abelian category
        $\QCoh(\Spec S,\alpha)$; recall that an object $M$
        of $\QCoh(\Spec S,\alpha)$ is finitely presented if
        $\Hom_{\QCoh(\Spec S,\alpha)}(M,-)$ commutes with all small coproducts. Here
        $\sperf$ refers to the fact that these are strictly perfect complexes.

        The homotopy category $\Ho\left(\Ch_{dg}^{\sperf}(\QCoh(\Spec S,\alpha))\right)$ is equivalent to the
        triangulated category of perfect complexes of $\alpha$-twisted $S$-modules, as every
        such perfect complex is quasi-isomorphic to a strictly perfect complex, that is a
        bounded complex of finitely presented projective objects.

        The pullback maps $f^*$ in the stacks $\StQCoh(\alpha)$ preserve finitely presented
        projective objects, so they induce dg functors
        \begin{equation*}
            f^*:\Ch_{dg}^{\sperf}(\QCoh(\Spec S,\alpha))\rightarrow\Ch_{dg}^{\sperf}(\QCoh(\Spec T,\alpha)).
        \end{equation*}
        The choice of higher coherences for the stack $\StQCoh(\alpha)$ as above yields the
        following information: (1) for each Zariski open immersion $\Spec S\rightarrow X$ a
        small $S$-linear dg category $\Ch_{dg}^{\sperf}(\QCoh(\Spec S,\alpha))$; (2) for
        each map $f:\Spec T\rightarrow\Spec S$ of Zariski open immersions in $X$ a dg functor
        \begin{equation*}
            f^*:\Ch_{dg}^{\sperf}(\QCoh(\Spec S,\alpha))\rightarrow\Ch_{dg}^{\sperf}(\QCoh(\Spec T,\alpha))
        \end{equation*}
        of small $S$-linear dg categories; (3) strict higher coherences encoding the
        associativity of composition, induced by those for $\StQCoh(\alpha)$. 

        We write $\StCh_{dg}^{\sperf}(\alpha)$ for the resulting prestack of small dg categories on
        $\Aff_X^{\Zar}$. The equivalence $\StQCoh(\alpha)\we\StQCoh(\beta)$ and its
        compatibility with the choice of higher coherence data for composition yields an
        equivalence $\StCh_{dg}^{\sperf}(\alpha)\we\StCh_{dg}^{\sperf}(\beta)$.

        Now, we can take right modules over these stacks to obtain a stack of big dg categories as follows.
        Recall that given a small $R$-linear dg category $\Crm_{dg}$, the dg category of
        right $\Crm_{dg}$-modules is the dg functor category $\Mod_{\Crm_{dg}}$ of maps
        $\Crm_{dg}^{\op}\rightarrow\Ch_{dg}(R)$, where $\Ch_{dg}(R)$ is the dg category of
        all chain complexes of $R$-modules. We further let $\Mod_{\Crm_{dg}}^c$ denote the
        full sub-dg category of cofibrant right $\Crm_{dg}$-modules, where we use the
        projective Quillen model category structure on $\Mod_{\Crm_{dg}}$.
        To any dg functor
        $\Crm_{dg}\rightarrow\Drm_{dg}$, we obtain a map
        $\Mod_{\Drm_{dg}}\rightarrow\Mod_{\Crm_{dg}}$, which preserves fibrations and
        quasi-isomorphisms. It is the right-adjoint of a quillen pair, so the left adjoint
        $\Mod_{\Crm_{dg}}\rightarrow\Mod_{\Drm_{dg}}$ preserves cofibrant objects. Moreover,
        since every representable module is cofibrant, the induced dg functor
        $\Mod_{\Crm_{dg}}^c\rightarrow\Mod_{\Drm_{dg}}^c$ is
        compatible with the
        Yoneda embeddings $\Crm_{dg}\rightarrow\Mod_{\Crm_{dg}}^c$ and
        $\Drm_{dg}\rightarrow\Mod_{\Drm_{dg}}^c$. Note that the homotopy category of
        $\Mod_{\Crm_{dg}}^c$ is, by definition, the derived category of $\Crm_{dg}$. When
        $\Crm_{dg}=\Ch_{dg}(R)$ for example, this construction coincides with the usual
        derived category of $R$.

        Hence, set $\Drm_{dg}^{\Zar}(\Spec S,\alpha)=\Mod_{\Ch_{dg}^{\sperf}(\Spec S,\alpha)}^c$.
        Thanks to the compatibilities discussed above,
        we obtain in this way a stack of locally presentable dg categories
        $\Dscr_{dg}^\Zar(\alpha)$, i.e. a prestack
        satisfying~\cite{toen-derived}*{Definition 3.6}. A couple words about this are in
        order: first the fact that it takes values in locally presentable dg categories
        follows from the fact that $\Mod_{\Ch_{dg}^{\sperf}(\Spec S,\alpha)}$ is compactly
        generated, say by the representable objects, and that the hom sets of the homotopy
        category are sets. Second, the fact that it is a stack
        follows for instance by identifying $\Drm_{dg}^\Zar(\alpha)$ with the restriction of
        the \'etale stack To\"en writes as $L_\alpha$ to $\Aff_X^{\Zar}$.
        Indeed, in To\"en's paper, for $\alpha\in\Br(X)$, there is a uniquely
        determined \'etale stack of locally presentable dg categories on $X$, denoted by
        $L_\alpha$; it is precisely the stack of dg categories of complexes of
        $\alpha$-twisted sheaves with quasi-coherent cohomology.
        The strict Zariski stack of small dg
        categories $\StCh_{dg}^\sperf(\alpha)$ is a model for the stack of compact objects
        in the restriction of $L_\alpha$ to $\Aff_X^{\Zar}$, and hence
        $\Drm_{dg}^\Zar(\alpha)$ is a model for the restriction of $L_\alpha$ to $\Aff_X^\Zar$.

        Our equivalence
        $\StCh_{dg}^{\sperf}(\alpha)\we\StCh_{dg}^{\sperf}(\beta)$ yields an equivalence
        $\Dscr_{dg}^{\Zar}(\alpha)\we\Dscr_{dg}^{\Zar}(\beta)$. It now remains to get an
        equivalence of \'etale stacks out of this equivalence of Zariski stacks.
        But, for any \'etale map $g:\Spec T\rightarrow X$, we can define a big $T$-linear dg
        category by
        \begin{equation*}
            \Drm_{dg}(\Spec T,\alpha)=\mathrm{holim}_{\Spec
            S\rightarrow\im(g)}\Drm_{dg}(\Spec S,\alpha)\widehat{\otimes}_S T,
        \end{equation*}
        where the homotopy limit is over all Zariski open immersions into $X$ factoring
        through the (open) image of $g$, and where the tensor product is defined as the
        dg category of cofibrant right modules over the small dg category $\Ch_{dg}(\Spec S,\alpha)\otimes_S T$.
        This construction is functorial in \'etale maps,
        and hence leads to an \'etale stack $\Dscr_{dg}^{\et}(\alpha)$. Moreover, the
        equivalence $\Dscr_{dg}^\Zar(\alpha)\we\Dscr_{dg}^\Zar(\beta)$ induces equivalences
        between the homotopy limit diagrams, and hence an equivalence
        $\Dscr_{dg}^\et(\alpha)\we\Dscr_{dg}^\et(\beta)$.

        By the adjunction between Zariski stacks and \'etale stacks, it
        follows that $\Dscr_{dg}^\et(\alpha)\we L_\alpha$ and $\Dscr_{dg}^\et(\beta)\we
        L_\beta$. Since
        $\Dscr_{dg}^{\et}(\alpha)\we\Dscr_{dg}^{\et}(\beta)$, it follows that there is an equivalence of \'etale
        stacks $L_\alpha\we L_\beta$. By the uniqueness of $L_\alpha$
        from~\cite{toen-derived}*{Corollary 3.12}, it follows that $\alpha=\beta$.
    \end{proof}
\end{theorem}

\begin{corollary}
    Suppose that $X$ and $Y$ are quasi-compact and quasi-separated schemes over a
    commutative ring $R$ with $\alpha\in\Br(X)$ and $\beta\in\Br(Y)$. If
    $\QCoh(X,\alpha)\we\QCoh(Y,\beta)$ as $R$-linear abelian categories, then there exists an
    isomorphism $f:X\rightarrow Y$ of $R$-schemes such that $f^*(\beta)=\alpha$.
\end{corollary}

\begin{corollary}
    Suppose that $X$ and $Y$ are noetherian schemes over a
    commutative ring $R$ with $\alpha\in\Br'(X)$ and $\beta\in\Br'(Y)$. If
    $\Coh(X,\alpha)\we\Coh(Y,\beta)$ as $R$-linear abelian categories, then there exists an
    isomorphism $f:X\rightarrow Y$ of $R$-schemes such that $f^*(\beta)=\alpha$.
\end{corollary}

\begin{remark}
    We can take $R=\ZZ$ in the above statements, in which case the condition is simply that
    the abelian categories involved be equivalent as abelian categories.
\end{remark}

\section{Concluding remarks}

\begin{enumerate}
\item
We do not know whether to expect the main theorem to hold for $\alpha\in\Br'(X)$ and
$\beta\in\Br'(Y)$. It might even be possible for it to hold for
$\alpha\in\Hoh^2_{\et}(X,\Gm)$ and
$\beta\in\Hoh^2_{\et}(Y,\Gm)$. That is, for arbitrary $\Gm$-gerbes. This would amount to a reconstruction
theorem for $\Gm$-gerbes.

\item
The passage through the work of To\"en while very satisfying seems to come out of the blue.
It would be nice to have a theory of Morita theory of stacks of abelian categories that is
internal in some sense to the theory of abelian categories\footnote{In forthcoming work, J.
Calabrese and M. Groechenig show that our main theorem \emph{does} hold for arbitrary
$\Gm$-gerbes, and they do so without recourse to derived algebraic geometry.}.

\item
At the moment, it seems like a much more difficult question to determine when
$\Drm_{\qc}(X,\alpha)\we\Drm_{\qc}(X,\beta)$ as $R$-linear triangulated categories. Examples
are given in~\cite{caldararu-elliptic} where $\alpha$ and $\beta$ are not related by any
automorphism of $X$. However, as in the proof above, the natural context for Morita
theory of Azumaya algebras is that of Morita equivalence of stacks. To\"en's theorem says
exactly that two Azumaya algebras have the same Brauer class if and only if the stacks of
dg categories of modules over $X$ are derived Morita equivalent as stacks. Asking when $\Drm_{\qc}(X,\alpha)\we\Drm_{\qc}(X,\beta)$
is like asking when two sheaves of $\Oscr_X$ modules have isomorphic
$\Gamma(X,\Oscr_X)$-modules of global sections, a mostly unnatural question.
\end{enumerate}

\end{document}